\documentclass[12pt]{amsart}

\usepackage[hmargin=2.5cm,bmargin=2.5cm,tmargin=3cm]{geometry}
\usepackage{soul}

\usepackage{amsmath,amssymb,amsthm,amsfonts,enumerate,url}
\usepackage{mathrsfs}
\usepackage{graphicx}
\usepackage{color}
\usepackage{float}

\usepackage{tikz}
\usetikzlibrary{automata,positioning,arrows.meta}

\usepackage[utf8]{inputenc}

\theoremstyle{plain}
\newtheorem{theorem}{Theorem}[section]
\newtheorem{lemma}[theorem]{Lemma}
\newtheorem{proposition}[theorem]{Proposition}
\newtheorem{corollary}[theorem]{Corollary}

\theoremstyle{definition}
\newtheorem{definition}[theorem]{Definition}
\newtheorem{example}[theorem]{Example}

\newtheorem{remark}[theorem]{Remark}

\numberwithin{equation}{section}

\usepackage{hyperref}

\newcommand\Graph{\mathcal{G}}

\newcommand\NewGraph{\widetilde{\mathcal{G}}}
\newcommand\ReflGraph{\widehat{\mathcal{G}}}

\newcommand{\lambN}{\lambda_2^{\mathrm{N}}}
\newcommand{\lambD}{\lambda_1^{\mathrm{D}}}
\newcommand{\lambO}{\lambda_1^{\mathrm{N}}}
\newcommand{\lambAK}{\lambda_1^{\mathrm{S}}}
\newcommand{\GraphD}{\Graph^{\mathrm{D}}}

\newcommand{\VertexSet}{\mV}

\newcommand{\R}{\mathbb{R}}

 \def\mG{\mathsf{G}}

 \def\mV{\mathsf{V}}

 \def\mE{\mathsf{E}}

 \def\mK{\mathsf{K}}

 \def\mv{\mathsf{v}}
 \def\me{\mathsf{e}}
 
 \def\mw{\mathsf{w}}

\DeclareMathOperator{\dist}{dist}

\DeclareMathOperator{\Dom}{Dom}

\newcommand{\nd}[3]{\partial_\nu {#1}|_{#2}(#3)} 

\title[Impediments to diffusion in quantum graphs]{Impediments to
  diffusion in quantum graphs: geometry-based upper bounds on the
  spectral gap}



\author[G.~Berkolaiko]{Gregory Berkolaiko}
\author[J.~B.~Kennedy]{James B.~Kennedy}
\author[P.~Kurasov]{Pavel Kurasov}
\author[D.~Mugnolo]{Delio Mugnolo}

\address{Gregory Berkolaiko, Department of Mathematics, Texas A{\&}M University, College Station, TX 77843-3368, USA}
\email{gberkolaiko@tamu.edu}

\address{James B.~Kennedy, Grupo de F\'isica Matem\'atica, Faculdade de Ci\^encias, Universidade de Lisboa, Campo Grande, Edif\'{i}cio~C6, \mbox{P-1749-016} Lisboa, Portugal}
\email{jbkennedy@fc.ul.pt}

\address{Pavel Kurasov, Department of Mathematics, Stockholm University, SE-106 91 Stockholm, Sweden}
\email{kurasov@math.su.se}

\address{Delio Mugnolo, Lehrgebiet Analysis, Fakult\"at Mathematik und Informatik, Fern\-Univer\-si\-t\"at in Hagen, D-58084 Hagen, Germany}
\email{delio.mugnolo@fernuni-hagen.de}

\date{\today}

\thanks{The work of G.B. was partially supported by the NSF under
 grant DMS--1815075.  J.B.K.~was supported by the Funda{\c{c}}{\~a}o
 para a Ci{\^e}ncia e a Tecnologia, Portugal, via the program
 ``Investigador FCT'', reference IF/01461/2015, and project
 PTDC/MAT-PUR/1788/2020.  P.K. was partially supported by the Swedish
 Research Council (Grant 2020-03780).  D.M. was partially supported
 by the Deutsche Forschungsgemeinschaft (Grant 397230547).  All four
 authors were partially supported by the Center for Interdisciplinary
 Research (ZiF) in Bielefeld, Germany, within the framework of the
 cooperation group on ``Discrete and continuous models in the theory
 of networks''. The contribution of J.B.K., P.K., and D.M.\ is based
 upon work from COST Action 18232 MAT-DYN-NET, supported by COST
 (European Cooperation in Science and Technology),
 \url{www.cost.eu}.}

\begin{document}

\begin{abstract}
  We derive several upper bounds on the spectral gap of the Laplacian
  on compact metric graphs with standard or Dirichlet vertex
  conditions. In particular, we obtain estimates based on the length
  of a shortest cycle (girth), diameter, total length of the graph, as
  well as further metric quantities introduced here for the first
  time, such as the avoidance diameter. Using known results about
  Ramanujan graphs, a class of expander graphs, we also prove that
  some of these metric quantities, or combinations thereof, do not to
  deliver any spectral bounds with the correct scaling.
\end{abstract}

\maketitle

\section{Introduction and statement of the main results}

Quantum graphs, a common term for differential operators in
function spaces defined on metric graphs, often arise as limits or
approximations of physical problems on thin structures.  They have
been used to study evolution of free electrons in molecules
\cite{RueSch53}, light propagation in optical waveguides (both
analytically and experimentally) \cite{LawBiaYun20,LawKurBau20}, heat
and water flow in branched pathways \cite{Mug07,SarCarAnd14}, Brownian
motions in ramified structures \cite{KosPotSch12}, vibration in steel
frames \cite{LagLeuSch94,Mei19,BerEtt21}, and many other applied
questions.

The most well-understood operator is the Laplacian with standard
(alternatively known as Neumann--Kirchhoff) as well as Dirichlet
vertex conditions\footnote{Basic definitions in the theory of
  Laplacian on metric graphs, as well as some results we use in the
  proofs of this paper, are collected in
  Appendix~\ref{sec:definitions_and_results}, where we also fix our
  notation.}.  On a connected graph and with no Dirichlet vertices
present, relaxation time of a diffusive process to equilibrium is
controlled by the first non-zero eigenvalue, which we denote by
$\lambN > \lambO = 0$.  In the presence of at least one Dirichlet
vertex,\footnote{Here and in the following, in a slight abuse of
  notation we will sometimes refer to properties of the (Dirichlet or
  standard) Laplacian as properties of $\Graph$, e.g.\ speak of
  Dirichlet vertices of $\Graph$.} a corresponding role is played by
the first eigenvalue $\lambD>0$, which controls dissipation of heat
from the graph.  These are the eigenvalues that we focus on in this
paper, aiming to give upper bounds in terms of certain metric
characteristics of the graph.

This study is motivated by the following question: what (geo)metric feature
can, on its own, act as an impediment to diffusion? 
To illustrate this question, consider the following.

\begin{example}\label{exa:intro}
  The simple estimates below only depend on the fact that certain
  specific metric graphs are (or are not) embedded in the given
  ambient metric graph $\Graph$ with vertex set $\mV$ and edge lengths
  $\ell_\me$, $\me\in\mE$:
  \begin{enumerate}
  \item If $\ell_{\max}$ is the length of the longest edge of $\Graph$, then
    \begin{equation}
      \label{eq:edge_length_estimate}
      \lambN \leq \frac{4\pi^2}{\ell_{\max}^2}.
    \end{equation}
    For a graph with Dirichlet vertices, the analogous estimate reads
    \begin{equation}
      \label{eq:edge_length_estimateD}
      \lambD \leq \frac{\pi^2}{\ell_{\max}^2}.
    \end{equation}
  \item Assume $\Graph$ contains neither a complete graph on five
    vertices ($\mK_5$), nor a complete bipartite graph on $3+3$
    vertices ($\mK_{3,3}$), as an induced
    subgraph{\footnote{That is, a subgraph formed from a
        subset of the vertices of $\Graph$ and all the edges of $\Graph$
        that connect the vertices in this subset.}}. Then
    \begin{equation}
      \label{eq:plum-plan}
      \lambN \leq \frac{16\pi^2}{L} \max_{\mv\in \mV}\sum_{\me\in\mE_\mv}\frac{1}{\ell_\me},
    \end{equation}
    where $L$ is the total length of $\Graph$ and $\mE_\mv$ is the set of 
    edges incident with $\mv$.
  \item \label{item:pumpkin-blockage} If a {cycle $\mathfrak{c}$}
    is included in $\Graph$ as an induced subgraph, and if { (the 
    closure of) each connected component of $\Graph \setminus \mathfrak{c}$ 
    meets $\mathfrak{c}$ at \emph{exactly one point},} then
    \begin{equation}
      \label{eq:pumpkin-blockage}
      \lambN (\Graph) \leq \frac{4\pi^2}{s^2}
    \end{equation}
    where $s$ denotes the { length of $\mathcal{S}$. Actually, 
    the same result holds if $\mathfrak{c}$ is a so-called \emph{pumpkin graph} 
    (see Figure~\ref{fig:pumpkin} below), and $s$ is the length 
    of the shortest cycle in it}.
  \end{enumerate}

  The first inequality is standard and will be shown in
  Section~\ref{sec:proofs_hybrid} below; inequality
  \eqref{eq:plum-plan} has been proved in~\cite[Theorem~3.11]{Plu21}
  and depends on the fact that, by Kuratowski's Theorem, a graph is
  planar if and only if it does not include any subgraph isomorphic to
  $\mK_5$ or $\mK_{3,3}$ (whereas \cite[Theorem~4.8]{Plu21} suggests
  that metric graphs of higher genus have higher $\lambN$); finally,
  the proof of inequality \eqref{eq:pumpkin-blockage} is based on the
  principle that attaching pendants to a graph lowers its eigenvalues
  (see \cite[Theorem~3.10]{BerKenKur19}), together with an estimate on
  the eigenvalues of the pumpkin graph.  This inequality also has a
  counterpart for higher eigenvalues.
\end{example}

The above examples suggest that having to cross a long edge, or an
``independent'' cycle, retards convergence to equilibrium. We are
going to make this observation more systematic.  Not to be overly
ambitious, we can try to use the length of the \emph{shortest}
cycle\footnote{This length is known in combinatorial graph theory as
  the girth, see also \cite[Section~6]{BerKenKur19}.} in place of
$\ell_{\max}$ in \eqref{eq:edge_length_estimate}: in the case of
standard vertex conditions, it is to be expected that the presence of
a ``minimal'' cycle of a given length should, like a long edge, be an
obstacle to rapid convergence.

In a graph with Dirichlet vertices, which act as heat sinks at any
point at which they are placed, the distance to the nearest Dirichlet
vertex becomes important, and the shortest distance between two such
vertices plays the same role as the minimal cycle length. For this
reason, for the purpose of defining girth (and only for this purpose,
cf.\ Remark~\ref{rem:dummy}), we identify all Dirichlet vertices; this
leads to the following modified definition:

\begin{definition}
  \label{def:dirichlet-girth}
  The \emph{girth} $s=s(\Graph)$ of a compact, connected graph
  $\Graph$ shall be given by
  \begin{displaymath}
    \min \{|\mathfrak{c}| \colon \mathfrak{c} \subset \Graph \text{ is
      a cycle in $\Graph'$} \},
  \end{displaymath}
  where $\Graph'$ is the metric graph obtained from $\Graph$ by
  identifying (or ``gluing together'') all Dirichlet vertices of
  $\Graph$, if any are present.  The girth is defined to be zero if
  $\Graph'$ is a tree.  Here $|\mathcal{H}|$ denotes the total length
  of a given subset $\mathcal{H}$ of $\Graph$.
\end{definition}

\subsection{Estimates based on girth}
\label{sec:walking-alone}

If the graph has at least one Dirichlet vertex, girth by itself is
indeed enough to yield an upper bound.

\begin{theorem}
  \label{thm:dirichlet-girth}
  If a compact, connected graph $\Graph$ has at least one
  Dirichlet vertex, then
  \begin{equation}
    \label{eq:dirichlet-girth}
    \lambD(\Graph) \leq \frac{\pi^2}{s^2}.
  \end{equation}
  Equality is attained if and only if $\Graph$ is an equilateral star
  graph with $n\geq2$ edges of length $s/2$ with Dirichlet conditions
  at all degree one vertices.
\end{theorem}

(The case of a Dirichlet interval of length $s$ corresponds to $n=2$.)

We believe Theorem~\ref{thm:dirichlet-girth} to be new even in the
case where $\Graph$ is a tree equipped with Dirichlet conditions at
all leaves. This may be contrasted with \cite[Eq.~(1.4)]{Roh17}, which
shows that for a tree graph, the diameter alone is enough to control
$\lambN$.


The result of Theorem~\ref{thm:dirichlet-girth} can also be immediately
extended to graphs with a particular reflection symmetry.  

\begin{corollary}\label{cor:girth-symm-neum}
  Suppose the compact, connected graph $\Graph$ is obtained from two copies
  of another connected graph, $\ReflGraph$, by pairwise gluing of finitely
  many pairs of the duplicated vertices. If $\Graph$ has girth $s$,
  then the spectral gap of the Laplacian with standard vertex
  conditions satisfies
  \begin{equation}
    \label{eq:waist-a-wish}
    \lambN(\Graph) \leq \frac{4\pi^2}{s^2}.
  \end{equation}
  Equality is attained if $\Graph$ is an equilateral pumpkin
  (watermelon) graph { (see Figure~\ref{fig:pumpkin}).}
\end{corollary}

{
\begin{figure}[H]
\centering
\begin{tikzpicture}[scale=0.75]
\coordinate (a) at (-2.5,0);
\coordinate (b) at (2.5,0);
\draw[thick] (a) -- (b);
\draw[thick,bend right]  (a) edge (b);
\draw[thick,bend left=70]  (a) edge (b);
\draw[thick,bend right=70]  (a) edge (b);
\draw[thick,bend left]  (a) edge (b);
\draw[fill] (a) circle (1.75pt);
\draw[fill] (b) circle (1.75pt);
\end{tikzpicture}
\caption{An \emph{$m$-pumpkin graph} (a.k.a.\ \emph{$m$-watermelon graph}) 
consists of some number $m\geq 2$ of parallel edges stretched between two vertices; 
here $m=5$. It is equilateral if all edges have the same length.}
\label{fig:pumpkin}
\end{figure}
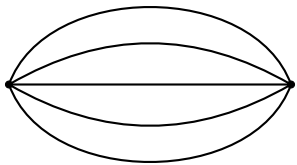
}

Note that $\Graph$ does not have any Dirichlet vertices, and so its
girth coincides with the length of its shortest cycle, which either
comes from a cycle in $\ReflGraph$ or from a cycle created in the
gluing process.  Figure~\ref{fig:reflection_graphs} presents an
example of $\ReflGraph$ covered by
Corollary~\ref{cor:girth-symm-neum}, and examples of symmetric graphs
which are not covered.

\begin{figure}[t]
  \centering
  \includegraphics[scale=0.75]{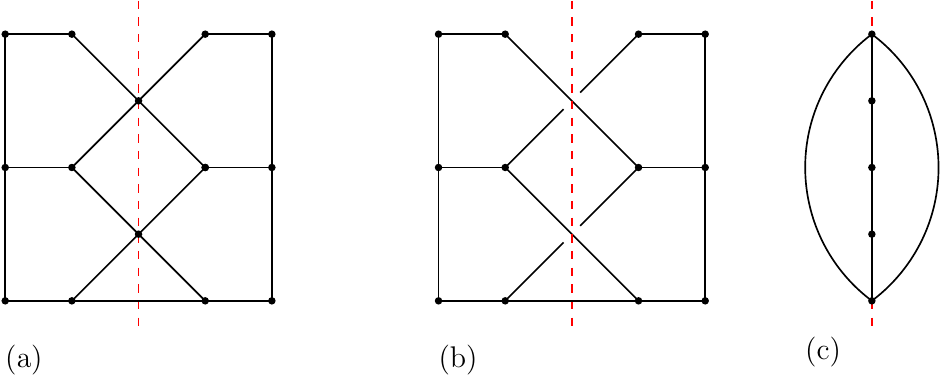}
  \caption{(a) An example of a graph that has the symmetry sufficient
    for \eqref{eq:waist-a-wish} to hold (note that we can place a
    dummy vertex in the middle of the lowest edge). (b) A graph with a
    reflection symmetry that cannot be obtained by gluing pairs of
    duplicated vertices. (c) A graph with a symmetry that would
    require gluing along entire edges.}
  \label{fig:reflection_graphs}
\end{figure}

Even without such symmetry, one could reasonably expect to use bound
\eqref{eq:dirichlet-girth} on each of the nodal domains of the second
standard Laplacian eigenfunction (generically, there are exactly two!)
to obtain \eqref{eq:waist-a-wish} for a general graph.  Furthermore,
\eqref{eq:waist-a-wish} can be shown to hold for many other classes of
``highly connected'' graphs, including \emph{pumpkin chains} { 
(finite sequences of pumpkin graphs glued together at their vertices to form 
a chain)}, equilateral complete graphs on at least three edges (see
\cite[Section~3]{KenKurMal16}) and also any graphs which 
contain an ``independent'' cycle, see
Example~\ref{exa:intro}(\ref{item:pumpkin-blockage}). 

It is thus all the more surprising that the girth alone is \emph{not} always 
enough to bound $\lambN$ from above, as the following examples show.

\begin{example}
  \label{ex:no-girth}
  We first provide a counterexample to the validity of
  \eqref{eq:waist-a-wish} with the symmetry restriction dropped. We consider Tutte's
  12-cage, a 3-regular graph on 126 vertices with girth 12 (cf., e.g.,
  \cite[p.~283]{ConMalMar06}): the second largest eigenvalue $\nu_2$
  of its adjacency matrix is known to be $\sqrt{6}$. A formula of von Below
  \cite[Theorem, p.~320]{Bel85}, which relates the eigenvalues of an
  equilateral metric graph to the eigenvalues of the normalized
  Laplacian, can be adapted to the eigenvalues of
  the adjacency matrix in the $k$-regular case, giving
  \begin{equation}
    \label{eq:vonbel}
    \lambN (\Graph)=\big(\arccos(\nu_2/k)\big)^2.
  \end{equation}
  We conclude that the spectral gap of the equilateral quantum graph
  built upon Tutte's 12-cage is $\big(\arccos(\sqrt{6}/3)\big)^2\approx 0.379$,
  whereas the right-hand side of~\eqref{eq:waist-a-wish} is
  $4\pi^2/144 \approx 0.274$.

  We remark that the automorphism group of Tutte's 12-cage is known to
  be a semi-direct product of PSU$(3,3)$ with the cyclic group 
  $\mathbf{Z}_2$; in particular the 12-cage graph does have
  a reflection symmetry.  It is far from obvious that this symmetry
  does not satisfy the assumptions of
  Corollary~\ref{cor:girth-symm-neum}.
\end{example}

\begin{example}
  \label{ex:LPSexample}
  Basic scaling arguments show that if an estimate of the form
  $\lambN(\Graph) \leq C s^\alpha$ is to exist, the power $\alpha$
  must be equal to $-2$.  We now show that such estimate is in fact
  impossible for any $C$. A counterexample is given by the class of
  equilateral metric graphs built out of combinatorial $k$-regular
  graphs known as Ramanujan graphs, which were introduced
  in~\cite{LubPhiSar88}. The second largest adjacency matrix
  eigenvalue $\nu_2$ of a Ramanujan graph is, by definition, no larger
  that $2\sqrt{k-1}$ (the largest one always being $k$, since the
  graph is $k$-regular), therefore
  \begin{equation}
    \label{eq:metricLPS}
   \lambN(\Graph) \geq \big(\arccos(2\sqrt{k-1}/k)\big)^2,
  \end{equation}
  again by \eqref{eq:vonbel}.
  It is established in~\cite{LubPhiSar88,BigBos90}, that for
  infinitely many values of $k$, Ramanujan graphs on arbitrarily many
  vertices $|\mV|$ can be constructed and that the asymptotic equality
  \begin{equation}
    \label{eq:girthLPS}
    s\sim \frac{4}{3}\log_{k-1}(|\mV|)
    \qquad \hbox{as }|\mV|\to\infty
  \end{equation}
  holds for these graphs. Combining \eqref{eq:metricLPS}
  and~\eqref{eq:girthLPS} we see that the estimate
  $\lambN(\Graph) = {\mathcal O}(s^{-2})$ cannot hold in general,  
  and thus no upper bound on $\lambN$ in terms of girth alone is possible.
\end{example}

\begin{remark}
  \label{rem:MSS_graphs}
  A more recent construction of Ramanujan graphs uses the theorem of
  \cite{MarSpiSri15} which, expressed in terms of metric graphs,
  states that any equilateral $k$-regular graph $\Graph$ has a
  \emph{signing} $\Graph_S$ such that its lowest eigenvalue satisfies
  \begin{equation}
    \label{eq:metricRamanujan}
   \lambAK(\Graph_S) \geq \big(\arccos(2\sqrt{k-1}/k)\big)^2.
  \end{equation}
  A signing is, in this context, a choice of edges $\{\me_j\}$ on
  which we impose anti-periodic conditions
  (see~\eqref{eq:AK_conditions_def} and
  Remark~\ref{rem:gauge_invariance}).  Let $\psi$ be the eigenfunction
  corresponding to $\lambAK$ of a graph $\Graph_S$ satisfying
  \eqref{eq:metricRamanujan}.  It is easy to see that the zeros of the
  function $\psi$ are located exactly on the anti-periodic cycles of
  the graph $\Graph_S$, i.e.\ the cycles on which an odd number of
  anti-periodic conditions have been imposed.  Imposing Dirichlet
  conditions at the locations of the zeros we obtain a graph $\GraphD$
  with only standard and Dirichlet conditions (by gauge invariance).
  Then $|\psi|$ is a non-negative eigenfunction of $\GraphD$ and thus
  the ground state (cf.\ Theorem~\ref{thm:changing_vc}); hence
  $\lambD(\GraphD)=\lambAK(\Graph_S)$.  The girth $s$ of $\GraphD$ is
  equal to the least of the two quantities: the length of the shortest
  cycle of $\Graph$ and the shortest distance between two zeros of
  $\psi$.  Theorem~\ref{thm:dirichlet-girth} then implies the estimate
  \begin{equation}
    \label{eq:girth_estimate_Ramanujan}
    s \leq \frac{\pi}{\arccos(2\sqrt{k-1}/k)}.
  \end{equation}
  The right-hand side is less than $10$ when $k=3$ and less than $3$
  when $k \geq 15$.  Intuitively, \eqref{eq:girth_estimate_Ramanujan}
  shows that the signing producing \eqref{eq:metricRamanujan} on a
  graph with large girth must be fairly dense (to guarantee a dense
  set of zeros of $\psi$).
\end{remark}

\subsection{Estimates based on total length and one further metric quantity}
\label{sec:hybrid}

Having seen that the ``single metric quantity estimates'' of the type
given in Theorem~\ref{thm:dirichlet-girth} are rather subtle, we will
now present some estimates that use a combination of two metric
quantities while having the correct overall scaling
$(\text{length})^{-2}$. For similar estimates involving other
quantities, we refer, e.g., to
\cite{BanLev17,BerKenKur19,BorCorJon21,KosNic19,Plu21,RohSei20} and
the references therein.

\begin{proposition}
  \label{thm:test-of-girth}
  If the compact, connected metric graph $\Graph$ has
  girth $s$ and total length $L$, then the spectral gap of the
  Laplacian with standard vertex conditions satisfies
  \begin{equation}
    \label{eq:girth-bound}
    \lambN (\Graph) < \frac{48L}{s^3}.
  \end{equation}
\end{proposition}

We remark that $s$ together with $L$ are also sufficient to bound
$\lambN$ from below, too \cite[Corollary~6.7]{BerKenKur19}.  The idea
behind estimate~\eqref{eq:girth-bound} (explained in
Section~\ref{sec:proofs_hybrid}) is to use a homotopy between two test
functions built around a minimal cycle; this idea readily generalizes
to combinations of $L$ with other metric quantities.

Recall that the \textit{diameter} $D$ of a metric graph $\Graph$ is
the maximal distance between any two points on the graph,
\begin{equation}
  \label{eq:diameter_def}
  D = \max_{x_1,x_2 \in \Graph} \dist(x_1,x_2).
\end{equation}

\begin{theorem}
  \label{thm:test-of-diameter}
  If the compact, connected metric graph $\Graph$ has
  diameter $D$ and total length $L$, then the spectral gap of the
  Laplacian with standard vertex
  conditions satisfies
  \begin{equation}
    \label{eq:diam-bound}
    \lambN (\Graph) < \frac{24L}{D^3}.
  \end{equation}
\end{theorem}

\begin{remark}
  \begin{enumerate}
  \item It has been known since \cite{KenKurMal16} that for a general
    $\Graph$ one cannot bound $\lambN$ in terms of $D$ alone.  The
    same work \cite[Theorem~7.1]{KenKurMal16} established the
    estimate
    \begin{equation}
      \label{eq:KKMM_LD_bound}
      \lambN(\Graph) \leq \frac{4\pi^2}{D^2}\left(\frac{L}{D} -
        \frac34\right).
    \end{equation}
    We also mention a generalization of Rohleder's
    diameter-based estimate~\cite{Roh17}, observed
    in~\cite{DufKenMug22}, namely
    \begin{equation}
      \label{eq:Rohleder_beta}
      \lambN(\Graph) \leq
      \frac{4\pi^2}{D^2}\left(\frac{1+\beta}2\right)^2,
    \end{equation}
    where $\beta$ is the first Betti number, i.e.,
    $\beta := |\mE| - |\mV| + 1$.  One can easily construct examples which
    show that none of the estimates \eqref{eq:girth-bound},
    \eqref{eq:diam-bound}, \eqref{eq:KKMM_LD_bound} or
    \eqref{eq:Rohleder_beta} is implied by a combination of the
    others.
  \item The Ramanujan graphs encountered in
    Example~\ref{ex:LPSexample} also show that we cannot in general
    improve \eqref{eq:girth-bound} to
    $\lambN(\Graph)\lesssim D / s^{3}$. The conclusion is
    obtained by combining the girth estimate~\eqref{eq:girthLPS} with
    the estimate $D \leq (2+\epsilon)\log_{k-1}(|\mV|)$ established in
    \cite{LubPhiSar88} (see also \cite{Sar19} for more precise
    diameter asymptotics in subfamilies of the LPS Ramanujan graphs).
  \end{enumerate}
\end{remark}

We now introduce some generalizations of the diameter, which
``interpolate'' between $D$ and $s$ (see \eqref{eq:ordering} and
Table~\ref{tab:metrics_comparison}).  We have not seen these
generalization in the existing literature, but they arise naturally
from the homotopy argument.  First, the \emph{triameter} $T$ of a
metric graph $\Graph$ is 
\begin{equation}
  \label{eq:triameter_def}
  T := \max_{x_1,x_2,x_3 \in \Graph} \min_{j\neq k} \, \dist(x_j,x_k),
\end{equation}
i.e.\ the maximal pairwise separation among any \emph{three} points on
$\Graph$.

\begin{theorem}
  \label{thm:test-of-triameter}
  If the compact, connected graph $\Graph$ has triameter $T$ and total
  length $L$, then the spectral gap of the Laplacian with standard
  vertex conditions satisfies
  \begin{equation}
    \label{eq:triam-bound}
    \lambN (\Graph) < \frac{12 L}{T^3}.
  \end{equation}
\end{theorem}

Second, we introduce the ``avoidance diameter'', which, intuitively,
measures how far apart two points can remain while exchanging places.
Let $\mathbb{S}^1$ be the unit circle in $\mathbb{C}$ and let $\Gamma$
denote the class of injective continuous maps from $\mathbb{S}^1$ to
$\Graph$.  Then the \emph{avoidance diameter} is defined as
\begin{equation}
  \label{eq:avoidance_diam_def}
  A := \max_{\gamma \in \Gamma} \min_{t\in\mathbb{S}^1}
  \dist\big(\gamma(-t), \gamma(t)\big).
\end{equation}
In the case of trees, $\Gamma=\emptyset$, and we set $A=0$.

\begin{theorem}
  \label{thm:avoidance_estimate}
  If the compact, connected graph $\Graph$ has avoidance diameter $A$
  and total length $L$, then
  \begin{equation}
    \label{eq:avoidance_estimate}
    \lambN(\Graph) < \frac{6 L}{A^3}.
  \end{equation}
\end{theorem}

\begin{remark}
  \label{rem:metrics_comparison}
  The metric quantities used in this section satisfy
  \begin{equation}
    \label{eq:ordering}
    \frac{s}{2} \leq A \leq D
    \qquad \text{and} \qquad
    \frac{s}{3} \leq T \leq D.
  \end{equation}
  Proposition~\ref{thm:test-of-girth} is now an immediate corollary of
  Theorem~\ref{thm:avoidance_estimate}.
\end{remark}

\begin{example}
For each inequality in \eqref{eq:ordering} one can construct examples
of metric graphs where that inequality is strict (see Table \ref{tab:metrics_comparison}).

\begin{table}[ht]
  \centering
  \begin{tabular}{l|c|c|c|c|c}
    $\Graph$ & $s$ & $A$  & $T$ & $D$ & Best estimate \\
    \hline 
    \hbox{path graph} & 0 & 0 & $\frac{L}{2}$ & $L$ &
    \eqref{eq:diam-bound}: $D$ \\[1.1pt] 
    \hline 
    \hbox{equilateral figure-8 graph} & $\frac{L}{2}$ & $\frac{L}{4}$ & $\frac{L}{4}$ & $\frac{L}{2}$ &
    \eqref{eq:diam-bound}: $D$ \\[1.1pt] 
    \hline 
    \hbox{figure-8 graph with lengths $\ell_1 > \ell_2 \geq \ell_1/2$}
    & $\ell_2$ & $\frac{\ell_1}2$
    & $\frac{\ell_1}2$ & $\frac{L}{2}$ &
    \eqref{eq:avoidance_estimate}: $A$ \\[1.1pt] 
    \hline 
    \hbox{equilateral flower graph on $k$ edges, $k\ge 3$} &
    $\frac{L}{k}$ & $\frac{L}{2k}$ & $\frac{L}{k}$ & $\frac{L}{k}$
    & \eqref{eq:triam-bound}: $T$ \\[1.1pt] 
    \hline 
    \hbox{equilateral star graph on $k$ edges, $k\ge 3$} & 0& 0
    & $\frac{2L}{k}$ & $\frac{2L}{k}$
    & \eqref{eq:triam-bound}: $T$ \\[1.1pt] 
    \hline 
    \hbox{equilateral pumpkin graph on $k$ edges, $k\ge 3$} &
    $\frac{2L}{k}$ & $\frac{L}{k}$ & $\frac{L}{k}$ & $\frac{L}{k}$
    & \eqref{eq:girth-bound} and \eqref{eq:avoidance_estimate}: $s$,
    $A$ \\ 
  \end{tabular} \\[4pt]
  \caption{A comparison of the various metric quantities defined in
    Section~\ref{sec:hybrid}.  We also show which estimate among those
  presented in Section~\ref{sec:hybrid} is the sharpest in each case.}
  \label{tab:metrics_comparison}
\end{table}

\end{example}

\begin{remark}
  Let $\mG$ be an unweighted combinatorial graph without loops or
  parallel edges, and consider the corresponding metric graph $\Graph$
  with all edges having length 1.  The \emph{normalized Laplacian}
  associated with $\mG$ is the $|\mV|\times |\mV|$-matrix whose
  diagonal entries are 1 and whose off-diagonal $(\mv,\mw)$-entry is
  $-\left(\deg(\mv)\deg(\mw)\right)^{-1/2}$, where $\deg(\mv)$ denotes
    the degree of vertex $\mv$; see~\cite[Section~1.2]{Chu97}.  It now
    follows from~\cite[Theorem, page 320]{Bel85}, already mentioned in
    Example~\ref{ex:no-girth}, that the lowest positive eigenvalue
    $\lambda_2$ of the standard Laplacian
    satisfies 
    \begin{equation}\label{eq:below}
      \alpha_2=1-\cos\sqrt{\lambda_2}\qquad \hbox{whenever }\lambda_2<
      \pi^2\ ,
\end{equation}
where $\alpha_2$ is the lowest positive eigenvalue of the normalized Laplacian. The quantities $s,D,T$ all have a natural version $s_\mG,D_\mG,T_\mG$ for combinatorial graphs, which are no larger than their counterpart on metric graphs. Accordingly, estimates~\eqref{eq:waist-a-wish}, \eqref{eq:girth-bound}, \eqref{eq:diam-bound}, and~\eqref{eq:triam-bound} all imply corresponding estimates on $\alpha_2$:
\begin{itemize}
\item $\alpha_2\le 1-\cos\frac{2\pi }{s_\mG}$ whenever $\mG$ is obtained from two copies of another connected graph, $\hat{\mG}$, by pairwise gluing of  pairs of the duplicated vertices;
\item $\alpha_2<1-\cos\sqrt{\frac{48 |\mE|}{s_\mG^3}}$ whenever $\frac{48 |\mE|}{s_\mG^3}<\pi^2$;
\item $\alpha_2<1-\cos\sqrt{\frac{24 |\mE|}{D_\mG^3}}$ whenever $\frac{24 |\mE|}{D_\mG^3}<\pi^2$;
\item $\alpha_2<1-\cos\sqrt{\frac{12 |\mE|}{T_\mG^3}}$ whenever $\frac{12 |\mE|	}{T_\mG^3}<\pi^2$.
\end{itemize}
To the best of our knowledge, upper estimates on $\alpha_2$ based on the total number of edges, on girth, diameter (and of course on triameter, which we have introduced in this paper) were not previously known.
\end{remark}

\section{Proofs}
\label{sec:proofs}

\subsection{Proof of the estimates in Section~\ref{sec:walking-alone}}
\label{sec:proof-of-girth}

In this section we prove Theorem~\ref{thm:dirichlet-girth} and
Corollary~\ref{cor:girth-symm-neum}.

\begin{proof}[Proof of Theorem~\ref{thm:dirichlet-girth}]
We want to show that if $\Graph$ has at least one Dirichlet
vertex, then the first non-trivial eigenvalue satisfies the bound
\eqref{eq:dirichlet-girth}, which we repeat here for convenience,
\begin{equation*}
  \lambD(\Graph) \leq \frac{\pi^2}{s^2}.
\end{equation*}
Due to our definition of the girth, the statement of the theorem is
vacuous if $\Graph$ has no cycles and only one Dirichlet point.
Henceforth we exclude such graphs from consideration.  Then we may
assume that vertices of $\Graph$ have degree one if and only if they
are equipped with the Dirichlet condition.  Indeed, we may remove any
pendant edges with standard conditions since this operation increases
the eigenvalues (see \cite[Theorem 2]{KurMalNab13} or
\cite[Theorem~3.10]{BerKenKur19}); a Dirichlet condition imposed at a
vertex of degree $d$ separates the vertex into $d$ copies.

The proof will be based on the notion (see \cite{BerKenKur19}) of
\emph{cutting through vertices along the eigenfunction} $\psi$
associated with $\lambD$, that is, finding a simple subgraph
$\NewGraph$ within $\Graph$, which, when cut out of $\Graph$ and
equipped with non-positive $\delta$-potentials determined by $\psi$,
will have its first eigenvalue equal to $\lambD(\Graph)$ (and its
eigenfunction will be $\psi|_{\NewGraph}$). The graph $\NewGraph$ can
then be compared directly with an interval or a tadpole (lasso)
without $\delta$-potentials, yielding the inequality.

We recall that $\psi \geq 0$ has no local minima and that
the maxima are isolated because $\psi(x)>0$ implies
\begin{equation}
  \label{eq:concavity}
  \psi''(x)=-\lambD(\Graph)\psi(x)<0.
\end{equation}
For the purpose of this proof all internal (to an edge) local maxima
are to be considered vertices.  Moreover, without loss of generality
we assume there are no other vertices of degree two, see
Remark~\ref{rem:dummy}.  The process of finding the subgraph
$\NewGraph$ uses the following notion of \emph{serious points} with
respect to the eigenfunction $\psi$.

\begin{definition}
  \label{def:serious-point}
  Given a metric graph $\Graph$ and a function $0 \leq f \in D(\Delta)$, a
  vertex $\mv \in \mV$ of degree $d \geq 2$ shall be
  called a \emph{serious point} (of the function $f$) if $f(\mv) \neq 0$
  and there exist at least two edges $\me_1,\me_2 \sim \mv$ such that
  \begin{equation}
    \label{eq:going_down_edge}
    \partial_\nu f|_{\me_i} (\mv) \geq 0, \qquad i=1,2.
  \end{equation}
  Here the normal derivatives are taken pointing \emph{into} the
  vertex.
\end{definition}

We remark that by the Kirchhoff condition, see
\eqref{eq:current_conservation_condition}, every vertex is incident with
\emph{at least one} edge satisfying
condition~\eqref{eq:going_down_edge}.

Any local maximum of $f$ (as defined in the obvious way) is serious,
therefore the set of serious points of $\psi$ is a non-empty finite
subset of the compact graph $\Graph$.  Let $\mv_0 \in \mV (\Graph)$
denote a lowest serious point, that is, $\mv_0$ is serious and
$0< \psi(\mv_0) \leq \psi (\mv)$ for every serious point
$\mv \in \mV (\Graph)$ of $\psi$.  Denote by $\me_{1},\me_{-1}$ any
two edges incident with $\mv_0$ such that
$\partial_\nu \psi|_{\me_{\pm 1}} (\mv_0) \geq 0$. 

Now denote by $\mv_{1}$ the other vertex incident with $\me_{1}$. If
it is a Dirichlet vertex, we stop.  Otherwise, by
concavity~\eqref{eq:concavity} we have
$\psi (\mv_{1}) < \psi (\mv_0)$; hence, $\mv_{1}$ is not serious.  We
denote by $\me_2$ the unique edge $\me_{2}$ adjacent to $\mv_{1}$ such
that $\partial_\nu \psi|_{\me_{2}} (\mv_{1}) \geq 0$.  Repeating the
process with $\me_{2}$ and so on, we obtain a path
$\me_1, \mv_1, \me_2, \ldots, \me_{m}$ terminating at a Dirichlet
vertex.  We perform the same for $\me_{-1}$, constructing another path
$\me_{-1}, \mv_{-1}, \me_{-2}, \ldots, \me_{-n}$ through $\Graph$,
also terminating at a Dirichlet vertex.  We define the graph
$\NewGraph$ to be the union of these two paths.  It is either a path
in $\Graph$ joining two Dirichlet vertices, or at some point the two
paths leading from $\me_{1}$ and $\me_{-1}$ meet and $\NewGraph$ is a
tadpole ending at a single Dirichlet point, see
Figures~\ref{fig:serious} and~\ref{fig:tadpole}.
\begin{figure}[H]
  \centering
  \includegraphics[scale=1]{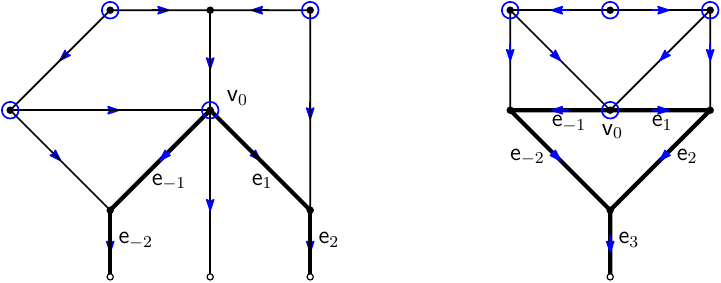}
  \caption{Examples of metric graphs with Dirichlet vertices (marked
    as small empty disks) and with a schematic depiction of the first
    eigenfunction as a gradient flow (blue arrows).  The serious
    points are now the vertices with at least two outward-pointing
    arrows; they are circled in blue.  The paths that are constructed
    in the course of the proof of Theorem~\ref{thm:dirichlet-girth}
    are shown as thicker edges.}
  \label{fig:serious}
\end{figure}
In the first case, the path is at least of
length $s$; in the second, the cycle of the tadpole is of
length at least $s$.
{
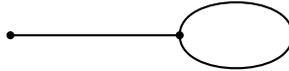
\begin{figure}[H]
\centering
\begin{tikzpicture}[scale=0.75]
\coordinate (a) at (-3,0);
\coordinate (b) at (0,0);
\coordinate (c) at (2,0);
\draw[thick] (a) -- (b);
\draw[thick,bend left=90]  (b) edge (c);
\draw[thick,bend right=90]  (b) edge (c);
\draw[fill] (a) circle (1.75pt);
\draw[fill] (b) circle (1.75pt);
\end{tikzpicture}
\caption{A \emph{tadpole} (or \emph{lasso}) consists of a cycle 
attached to a pendant edge. In the proof the degree one vertex will be 
equipped with a Dirichlet condition.}
\label{fig:tadpole}
\end{figure}
}
Now we wish to cut $\NewGraph$ out of $\Graph$ ``along'' the
eigenfunction $\psi$: at each non-Dirichlet vertex $\mv_{i}$ of
$\NewGraph$ ($i=0,\pm 1, \pm 2, \ldots$), we add a $\delta$-condition
of the form
\begin{displaymath}
  \sum_{\substack{
  \me\sim \mv_{i} \\
  \me \in \NewGraph}
  } \partial_\nu u|_{\me} (\mv_{i})
  + \gamma_{\mv_i} u(\mv_{i}) = 0,
\end{displaymath}
where
\begin{equation}
  \label{eq:Kirchhoff_leftover}
  \gamma_{\mv_i} := \frac1{\psi(\mv_{i})} \sum_{\substack{
  \me\sim \mv_{i}\\
   \me \in \Graph \setminus \NewGraph}}
  \partial_\nu \psi|_{\me} (\mv_{i}) = 
  - \frac1{\psi(\mv_{i})} \sum_{\substack{\me\sim \mv_{i} \\
   \me \in \NewGraph}} 
  \partial_\nu \psi|_{\me} (\mv_{i}), 
\end{equation}
where the second equality follows from the Kirchhoff
condition.  The $\delta$-conditions are chosen precisely so that
$\psi|_{\NewGraph}$ is still an eigenfunction of $\NewGraph$.
Therefore, $\lambD (\Graph)$ is still an eigenvalue, and, since $\psi$
is non-negative, $\lambD (\Graph) = \lambda_1(\NewGraph)$ (cf.\
Theorem~\ref{thm:changing_vc}).

We claim that $\gamma_{\mv_0} \leq 0$ and $\gamma_{\mv_i} < 0$ for
$i\neq 0$ (recall that $\nu$ points into $\mv_i$). For $\mv_0$ this
follows from the choice of $\me_{1}$ and $\me_{-1}$. For $i \neq 0$
this follows since $\mv_i$ is not serious: $\me_{i+1} \in \NewGraph$
is the only edge incident with $\mv_i$ in $\Graph$ for which
$\partial_\nu \psi|_{\me} (\mv_{i}) \geq 0$; hence all derivatives in
the first sum in \eqref{eq:Kirchhoff_leftover} are strictly negative.
Note that the degree of $\mv_i$ is at least 3 for all $i\neq0$ since all degree 2 vertices
have been suppressed; thus the sum contains at least one summand.

By Theorem~\ref{thm:changing_vc}, replacing
all $\gamma_{\mv_i}$ with 0 can only increase
the eigenvalues.  Therefore, $\lambD (\Graph)$ is bounded from above
by the first eigenvalue of either the Dirichlet interval of length $s$
or a Dirichlet tadpole with cycle length $s$.  In both cases the
eigenvalue is at most $(\pi/s)^2$ by estimate \eqref{eq:edge_length_estimateD}.

Finally, we discuss the case of equality.  The eigenvalue of an
equilateral Dirichlet star is well known to be $(\pi/2\ell)^2$, where
$\ell=s/2$ is the edge length.  To prove necessity, we first observe
that $\psi|_{\NewGraph}$ is a simple eigenfunction which does not
vanish at any point of $\NewGraph$ except for the Dirichlet
vertices. In particular, in the case of a tadpole the inequality must
be strict if the outgrowth is non-trivial.  In the case of a path,
$\NewGraph$ cannot contain any edges other than $\me_1$ or $\me_{-1}$,
since $\gamma_{\mv_i} < 0$ when $i\neq0$, and strictly increasing
$\gamma$ strictly increases $\lambda_1$
(Theorem~\ref{thm:changing_vc}). The same reasoning yields
$\partial_\nu \psi|_{\me_{\pm 1}} (\mv_0) = 0$.

If the degree of $\mv_0$ is larger than 2, the Kirchhoff
condition now implies
\begin{equation}
  \label{eq:Kirchhoff_leftover2}
  \sum_{\me\sim \mv_0, \me\neq \me_{\pm 1}} 
  \partial_\nu \psi|_{\me} (\mv_0) = 0,
\end{equation}
and therefore there is another edge $\widetilde{\me}_1$ with
$\partial_\nu \psi|_{\widetilde{\me}_1} (\mv_0) \geq 0$.  Repeating the
proof with $\widetilde{\me}_1$ instead of $\me_1$ we can similarly
conclude that $\widetilde{\me}_1$ leads to a Dirichlet vertex and
$\partial_\nu \psi|_{\widetilde{\me}_1} (\mv_0) = 0$.  Proceeding by
induction (and applying \eqref{eq:Kirchhoff_leftover2} to the sums
over fewer and fewer edges), we conclude that every edge incident with
$\me_0$ leads to a Dirichlet vertex with no vertices of degree 
$\geq 3$ along the way. 
Thus $\Graph$ is a star graph, whose ground state eigenfunction reaches its maximum 
at the central vertex $\mv_0$.  Moreover,
$\partial_\nu\psi_\me(\mv_0)=0$ for every incident edge $\me$, therefore
--- since the eigenvalue is $\pi^2/s^2$ --- the edge length must be $s/2$.
\end{proof}

\begin{proof}[Proof of Corollary~\ref{cor:girth-symm-neum}]
  Let $R$ be the reflection operator on $L^2(\Graph)$ induced by the
  corresponding reflection operator $\rho: \Graph \to \Graph$ acting
  on the metric space $\Graph$. Now, $R$ is a bounded linear operator on
  $L^2(\Graph)$ that commutes with the standard Laplacian; in
  particular, the standard Laplacian is reduced by the closed
  orthogonal subspaces
  \[
    {\mathcal S} := \left\{\frac{f+Rf}{2}:f\in L^2(\Graph)\right\}
    \quad\hbox{and}\quad
    {\mathcal A} := \left\{\frac{f-Rf}{2}:f\in L^2(\Graph)\right\}
  \]
  of symmetric and anti-symmetric functions in $L^2(\Graph)$,
  respectively.
  
  Since the ground state of $\Graph$ clearly belongs to $\mathcal{S}$,
  \begin{equation}
    \label{eq:antisymm_states}
    \lambda_2(\Delta_\Graph) \leq
    \lambda_1\left(\Delta_\Graph\big|_{\mathcal{A}}\right).    
  \end{equation}
  It is easy to see that ${\mathcal A}$ is isomorphic to
  $L^2({\ReflGraph})$, whereas
  $\Dom\left(\Delta_\Graph\big|_{\mathcal{A}}\right)$ is the domain of
  the standard Laplacian on $\ReflGraph$ with the exception of
  Dirichlet conditions at
  $\partial\ReflGraph = \{\mv \in \Graph: \rho(\mv) = \mv\}$,
  identified as a subset of $\ReflGraph$.  The Dirichlet conditions
  may decompose $\ReflGraph$ into several connected components joined
  together at $\partial\ReflGraph$. In any case, each connected
  component of $\ReflGraph$ necessarily has girth at least $s/2$,
  while the estimate \eqref{eq:dirichlet-girth} also applies on each
  connected component.  Estimate \eqref{eq:waist-a-wish} follows
  immediately.
\end{proof}

\subsection{Proof of the estimates in Section~\ref{sec:hybrid}}
\label{sec:proofs_hybrid}

Throughout this section we assume the graph $\Graph$ is finite,
compact, and connected, and has standard (Neumann--Kirchhoff)
conditions at every vertex.  All upper bounds in this section are
based on the variational characterization of the second eigenvalue,
\begin{equation}
  \label{eq:variational_second_eig}
  \lambN(\Graph) = \min\left\{
    \frac{\int_\Graph |f'(x)|^2 dx}{\int_\Graph |f(x)|^2 dx}
    \colon
    f \in H^1(\Graph),\ \int_\Graph f(x) dx = 0 \right\},
\end{equation}
where the Sobolev space $H^1$ of the graph is defined as
\begin{equation}
  \label{eq:Sobolev_H1_def}
  H^1(\Graph) := \left\{
    u\in C(\Graph)\cap \bigoplus_{e\in E}H^1(0,\ell_e)
    \colon
    \|u'\|_{L^2(\Graph)}<\infty \right\}.
\end{equation}
Upper bounds can now be obtained by choosing a suitable test
function $f \in H^1(\Graph)$ having mean value 0 (i.e. being 
orthogonal to the constants, which span the
eigenspace of the first eigenvalue).
For example, using the test function
$f(x) = \sin\left(2\pi x/\ell_{\max}\right)$ on the longest edge,
extended by zero to the rest of the graph, we immediately obtain
estimate \eqref{eq:edge_length_estimate} from the introduction.

The main difficulty in choosing the test function arises from the
requirement that $f$ have mean zero.  Our principal tool for satisfying
this requirement will be to build a homotopy between a function and
its negative in the punctured form domain.

\begin{lemma}
  \label{lem:homotopy-gen}
  Let $a$ be a positive, closed quadratic form whose domain $D(a)$ is compactly embedded in a Hilbert space $H$. Assume the associated positive semi-definite self-adjoint operator $A$ on  $H$ to have one-dimensional null space  spanned by some function $u$.
If, { for a family of functions $\psi_\cdot \colon [0,1] \to D(a) \setminus \{0\}$, 
\begin{itemize}
\item[(1)] the mapping $[0,1]\ni t \mapsto \left< \psi_t, u\right>_H\in \R$ is continuous, and 
\item[(2)] $\psi_0 = -\psi_1$,
\end{itemize}
 then} there exists $t_0$ such
  that the second lowest eigenvalue of $A$ satisfies
  \begin{equation}
    \label{eq:lambda2_est_homotopy-abs}
    \lambda_2(A) \leq \frac{a(\psi_{t_0})}
    {\left\|\psi_{t_0}\right\|_{H}^2}.
  \end{equation}
\end{lemma}

\begin{proof}
  Since
  $\left< \psi_0, u \right>_{H} = -\left< \psi_1, u
  \right>_{H}$, by the Intermediate Value Theorem there is
  at least one $t_0\in[0,1]$ such that
  $\left< \psi_{t_0}, u \right>_{H} = 0$.  Then we use
  $\psi_{t_0}$ as a test function in the abstract version of 
  \eqref{eq:variational_second_eig}.
\end{proof}

Taking $H=L^2(\Graph)$, $u=\mathbf 1 \in L^2(\Graph)$ and $a$ the
quadratic form associated with the metric graph Laplacian with
standard vertex conditions, we immediately obtain the following.
\begin{corollary}
  \label{cor:homotopy}
  Let $\psi_\cdot \colon [0,1] \to H^1(\Graph) \setminus \{0\}$ be such that
  $\psi_0 = -\psi_1$ and the mapping
  $t \mapsto \left< \psi_t, 1 \right>_{L^2(\Graph)}$ is a continuous
  function $[0,1] \to \R$.  Then there exists $t_0$ such
  that
  \begin{equation}
    \label{eq:lambda2_est_homotopy}
    \lambN (\Graph) \leq \frac{\left\|\psi_{t_0}'\right\|_{L^2(\Graph)}^2}
    {\left\|\psi_{t_0}\right\|_{L^2(\Graph)}^2}.
  \end{equation}
\end{corollary}

The proofs of the theorems in Section~\ref{sec:hybrid} are now reduced
to constructing a suitable family of test functions.


\begin{proof}[Proof of Theorem~\ref{thm:test-of-diameter}]
  Introduce the ``tent'' function
  \begin{equation}
    \label{eq:tent_test}
    \tau_{y, d}(x) =
    \begin{cases}
      d - \dist(x,y), & \hbox{if }\dist(x,y) \leq d, \\
      0, & \text{otherwise},
    \end{cases}
  \end{equation}
  let $d=D/2$, and take
  \begin{equation}
    \label{eq:homotorpy_for_diameter}
    \psi_t := \cos(\pi t) \tau_{x_1, d} + \sin(\pi t) \tau_{x_2, d},
  \end{equation}
  where $x_1$ and $x_2$ are a pair of points on the graph realizing
  the diameter.  Note that $\tau_{x_1, d}$ and $\tau_{x_2, d}$
  have disjoint supports.  Then $\psi_t$ satisfies the conditions of
  Corollary~\ref{cor:homotopy} and we can estimate
  \begin{equation*}
    \left\|\tau_{x_1,d}'(x)\right\|_{L^2(\Graph)}^2
    \leq \int_{x \colon \dist(x,x_1) \leq d} 1 dx
    \leq L,
  \end{equation*}
  and
  \begin{equation}
    \label{eq:diam_L2_est}
    \left\|\tau_{x_1,d}(x)\right\|_{L^2(\Graph)}^2
    \geq
    \int_{0}^{d} (d-x)^2 dx
    = \frac{d^3}3,
  \end{equation}
  where we estimate the $L^2$-norm by only integrating along the path
  realizing the diameter.  Combining, we obtain the desired estimate
  \begin{equation*}
    \lambN \leq
    \frac{\cos^2(\pi t)L + \sin^2(\pi t)L}{\cos^2(\pi t)\frac{D^3}{24}
      + \sin^2(\pi t)\frac{D^3}{24}} \leq 
    \frac{24 L}{D^3}.
  \end{equation*}
  Since our test function cannot possibly be an eigenfunction, being
  piecewise linear, the inequality is strict.
\end{proof}


\begin{proof}[Proof of Theorem~\ref{thm:avoidance_estimate}]
  We use Corollary~\ref{cor:homotopy} with
  \begin{equation}
    \label{eq:homotopy_avoidance}
    \psi_t := \tau_{\gamma(e^{i\pi t}), d} - \tau_{\gamma(-e^{i\pi t}), d},
  \end{equation}
  where $\gamma$ is the curve realizing the avoidance diameter and
  $d=A/2$.  The supports are
  disjoint and the gradient is 1 or 0, therefore
  $\|\psi_t'\|^2 \leq L$.  For the norm, we amend estimate
  \eqref{eq:diam_L2_est} to use injectivity of $\gamma$ and obtain
  \begin{equation}
    \label{eq:diam_L2_est_inject}
    \|\psi_t\|^2
    \geq
    2\int_{-d}^{d} \big(d-|x|\big)^2 dx
    = \frac{4d^3}3 = \frac{A^3}{6}.
  \end{equation}
  It is clear that the test function is not an eigenfunction,
  therefore the inequality is strict.
\end{proof}




\begin{proof}[Proof of Theorem~\ref{thm:test-of-triameter}]
  We use Corollary~\ref{cor:homotopy} with
  \begin{equation}
    \label{eq:homotopy_triameter}
    \psi_t := h_1(t) \tau_{x_1, d} + h_2(t) \tau_{x_2, d} + h_3(t) \tau_{x_3, d},
  \end{equation}
  where $x_1,x_2,x_3 \in \Graph$ are any three points realizing the triameter,
  $d=T/2$, and the functions $h_j$ satisfy $h_j(0)=1$ and $h_j(1)=-1$.
  Using the disjoint supports, we can estimate
  \begin{equation*}
    \lambN \leq \max_{t\in[0,1]} \frac{3L 
      \max \left\{h_1(t)^2, h_2(t)^2, h_3(t)^2\right\}}
    {d^3 \big(h_1(t)^2+h_2(t)^2+h_3(t)^2\big)}.
  \end{equation*}
  Choosing $h_j$ so that $|h_j(t)|\leq1$ and only one of them can be
  different from $\pm1$ at any given time (in other words, they take
  turns to go from 1 to $-1$), yields
  \begin{equation*}
    \lambN < \frac{3L}{2d^3} = \frac{12L}{T^3}.
  \end{equation*}
    This concludes the proof.
\end{proof}

\appendix 
\section{Laplacian on metric graphs: definitions and useful results}

\label{sec:definitions_and_results}

In this appendix we review the basic definitions and terminology used
in this paper.  For further information, the reader is invited
to consult various surveys and books on the subject
\cite{Berk17,BerKuc13,Mug15}.  We also formulate a 
result based on \cite{BerKenKur19,Kur19} that we use 
repeatedly in Section~\ref{sec:proof-of-girth}.

Let $\Graph = (\mV,\mE)$ be a graph with vertex set
$\mV$ and edge set $\mE$.  The graph is a metric graph
if each edge $\me\in \mE$ is identified with an interval $(0,\ell_\me)$, where $\ell_\me>0$ is regarded as the length of the edge.  We shall write
$\me \sim \mv$ to mean that the edge $\me$ is incident with the vertex
$\mv$; and denote by $\mE_\mv$ the set of such edges. A graph is \emph{compact}, if it has a finite number
of edges, each edge of finite length.  We denote by $L$ or $|\Graph|$
the total length of the graph, i.e.\ the sum of the lengths of the
edges of the graph.  $\Graph$ is allowed to contain loops as well as
multiple edges between given pairs of vertices. A cycle in $\Graph$
is, formally, a map $\gamma: [0,1] \to \Graph$ such that
$\gamma (0) = \gamma (1)$ and $\gamma$ is injective on
$[0,1)$. However, we do not usually distinguish between $\gamma$ and
its image, a closed subset of $\Graph$.

In this paper we study the spectrum of the Laplacian $-\Delta$ on
$\Graph$.  More precisely, the operator acts as $-\frac{d^2}{dx^2}$ on
the functions which are in the Sobolev space $H^2(0,\ell_\me)$ on each edge
$\me\in \mE$.  The domain $D(\Delta)$ of the operator is further restricted
to functions that satisfy at any vertex $\mv \in \mV$ one of
the following conditions:
\begin{itemize}
\item \emph{Standard}\footnote{Also known as Neumann--Kirchhoff, 
    among other names;
    observe that on a degree-one vertex standard conditions agree with
    common Neumann ones.} conditions: at $\mv$, we demand continuity of
  the functions and that the sum of the normal derivatives at each
  vertex is zero (``Kirchhoff'' or ``current conservation''
  condition):
  \begin{equation}
    \label{eq:current_conservation_condition}
    \sum_{\me \sim \mv} \nd{f}{\me}{\mv} = 0,
  \end{equation}
  where $\nd{f}{\me}{\mv}$ is the normal derivative of $f$ on $\me$ at $\mv$,
  with $\nu = \nu_\me(\mv)$ pointing \emph{outward} (away from the edge
  $e$, towards the vertex);
\item \emph{Dirichlet} conditions: at $\mv$, any
  functions in the domain of $\Delta$ should take on the value zero. 
  We denote the set of vertices equipped with Dirichlet conditions 
  by $\VertexSet_D$.
\item \emph{Anti-periodic}\footnote{The name
    ``anti-periodic'' goes back to the theory of Hill's operator, see,
    e.g., \cite{MarOst_ms75}.  Some authors use the name ``signing
    conditions'' to highlight connections to the construction of
    combinatorial Ramanujan graphs \cite{BilLin_c06,MarSpiSri15}. For
    vertices of general degree, the corresponding conditions are often
    called ``anti-Kirchhoff'' conditions, as in, e.g.,
    \cite[Classification~2.3.II]{RohSei20}.} conditions: at $\mv$ of
  degree 2,
  \begin{equation}
    \label{eq:AK_conditions_def}
    f_{\me_1}(\mv) = -f_{\me_2}(\mv), \qquad
    \partial_\nu f_{\me_1}(\mv) = \partial_\nu f_{\me_2}(\mv).
  \end{equation}
  { These conditions are used in
    Remarks~\ref{rem:MSS_graphs} of the
  present paper (see also Remark~\ref{rem:gauge_invariance} below.}
\item \emph{$\delta$ (or Kirchhoff--Robin)} conditions: associated with $\mv$
  there is a $\gamma = \gamma(\mv) \in \R$, $\gamma\neq 0$ such that the
  functions $f \in D(\Delta)$ are continuous at $\mv$ and the
  derivatives at $\mv$ satisfy
  \begin{equation}
    \label{eq:robin_def}
    \sum_{\me \sim \mv} \nd{f}{\me}{\mv} + \gamma f(\mv) = 0,
  \end{equation}
  where, again, $\nu$ is the outer unit normal to the edge. 
  We sometimes refer to $\gamma$ as the \emph{strength} of the
  $\delta$-condition, or as the \textit{$\delta$-potential} at $\mv$.
\end{itemize}
We see immediately that the $\delta$-condition with $\gamma(\mv)=0$
corresponds to the standard condition.  Furthermore, Dirichlet
conditions correspond formally to $\delta$-conditions of strength
$\gamma=\infty$ (this correspondence may be made rigorous
\cite{BerKuc12}, and we will use it below).

\begin{remark}
  \label{rem:dummy}
  Any point in the interior of an edge may be declared to be a vertex
  of degree two with standard conditions without affecting the spectral
  properties of the operator.  We will refer to this as introducing a
  ``dummy'' vertex.  Conversely, any vertex $v$ of degree two with
  standard conditions may be suppressed.  Likewise, the operator is
  not modified if a subset of the elements of $\mV_D$ are
  identified to form one single Dirichlet vertex.
\end{remark}

\begin{remark}
  \label{rem:gauge_invariance}
  Declaring an arbitrary point $x\in\me$ to be a vertex $\mv=\mv_x$ of
  degree two, we can also impose anti-periodic
  conditions~\eqref{eq:AK_conditions_def} there.  Due to 
  ``gauge invariance'', different choices for the location $x$  
  within the same edge result
  in unitarily equivalent operators \cite[Section 2.6]{BerKuc13};
  we thus refer to this as imposing anti-periodic conditions on
  the edge $\me\in\mE$.
\end{remark}


As is well known, under the above set of assumptions the Laplacian is
self-adjoint, semi-bounded, and has trace class resolvent; in
particular, its spectrum consists of a sequence of real eigenvalues of
finite multiplicity, which we denote by
\begin{equation}
  \label{eq:eigenvalues}
  \lambda_1 \leq \lambda_2 \leq \lambda_3 \leq \ldots,
\end{equation}
where each is repeated according to its multiplicity. The corresponding
eigenfunctions may be chosen to form an orthonormal basis of
$L^2(\Graph)$, and may additionally without loss of generality all be
chosen real.

The following theorem summarizes properties of the first eigenvalue 
used extensively in Section~\ref{sec:proof-of-girth}, including a 
an interlacing inequality for the eigenvalues, which in its sharpest form 
(with a characterization of equality) appeared in \cite{BerKenKur19}.

\begin{theorem}
\label{thm:changing_vc}
Let $\Graph$ be a compact graph equipped with 
a $\delta$-condition of strength $\gamma_i \in (-\infty,\infty]$ at each vertex 
$\mv_i \in \mV$. Denote by $\VertexSet_D$ the set of all Dirichlet vertices, 
\begin{displaymath}
	\VertexSet_D = \{ \mv_i \in \VertexSet : \gamma_i = \infty \},
\end{displaymath}
and suppose that $\Graph \setminus \VertexSet_D$ is connected. Then 
$\lambda_1 = \lambda_1 (\Graph)$ is simple and its eigenfunction $\psi$ may 
be chosen strictly positive in $\Graph \setminus \VertexSet_D$.

Moreover, if $\Graph'$ is formed by replacing the potential $\gamma_i \in \R$ at 
$\mv_i$ with $\gamma_i' \in (\gamma_i,\infty]$, then
\begin{equation}
\label{eq:surgery-eigenvalue-lifting}
	\lambda_1 (\Graph') > \lambda_1 (\Graph).
\end{equation}
\end{theorem}

\begin{proof}
For simplicity of $\lambda_1$ and strict positivity of $\psi$, see \cite{Kur19}. 
For \eqref{eq:surgery-eigenvalue-lifting}, since $\psi$ does not satisfy the 
$\delta$-condition at $\mv_i$ in the graph $\Graph'$, it cannot be an 
eigenfunction; hence $\lambda_1 (\Graph')$ and $\lambda_1 (\Graph)$ 
have no eigenfunctions in common. The (strict) inequality 
\eqref{eq:surgery-eigenvalue-lifting} now follows immediately from 
\cite[Theorem~3.4]{BerKenKur19}.
\end{proof}

\bibliographystyle{plain}
\bibliography{literatur}

\end{document}